\documentclass[12pt]{amsart}
\usepackage[pagebackref=false,hyperindex=true,colorlinks=false, pdfauthor={Olivier Haution}, pdftitle={On the first Steenrod square for Chow groups}, pdfkeywords={  {S}teenrod operations, {R}iemann-{R}och theorem, {C}how groups, Witt indices}]{hyperref}

\usepackage{enumerate,amssymb,version}
\usepackage{geometry}
\usepackage[all]{xy}

\DeclareMathOperator{\id}{id}
\DeclareMathOperator{\gr}{gr}
\DeclareMathOperator{\Spec}{Spec}

\newcommand{\Z}{\mathbb{Z}}
\newcommand{\Q}{\mathbb{Q}}
\newcommand{\iu}{\mathfrak{i}_1}

\newcommand{\ba}[1]{\overline{#1}}
\newcommand{\ti}[1]{\widetilde{#1}}

\newcommand{\Kzl}{K_0}
\newcommand{\Kit}[2]{\Kzl({#2})_{({#1})}}
\newcommand{\Grt}[2]{\gr_{{#1}}\Kzl({#2})}
\newcommand{\Grts}{\gr_{\bullet}\Kzl}

\DeclareMathOperator{\CH}{CH}
\DeclareMathOperator{\Ch}{Ch}
\DeclareMathOperator{\Cht}{\ti{Ch}}
\newcommand{\CHQ}[2]{\CH_{#1}({#2})_{\Q}}
\newcommand{\CHZ}[2]{\CH_{#1}({#2})_{\Z}}
\newcommand{\CHQs}[1]{\CH({#1})_{\Q}}
\newcommand{\CHZs}[1]{\CH({#1})_{\Z}}

\newcommand{\Fc}{\mathcal{F}}
\newcommand{\Ec}{\mathcal{E}}
\newcommand{\Ecd}{\mathcal{E}_{\bullet}}
\newcommand{\Oc}{\mathcal{O}}

\newcommand{\Tan}{T}
\newcommand{\PP}{\mathbb{P}}
\newcommand{\cc}{c}

\DeclareMathOperator{\Square}{Sq}
\newcommand{\Squu}{\Square_1}
\newcommand{\Squ}{\Square^1}

\DeclareMathOperator{\Todd}{Td}
\DeclareMathOperator{\ch}{ch}

\newtheorem{theorem}{Theorem}[section]
\newtheorem{proposition}[theorem]{Proposition}
\newtheorem{lemma}[theorem]{Lemma}

\theoremstyle{definition}
\newtheorem{remark}[theorem]{Remark}

\begin{document}
\begin{abstract}
We construct a weak version of the homological first Steenrod square, a natural transformation from the modulo two Chow group to the Chow group modulo two and two-torsion. No assumption is made on the characteristic of the base field. As an application, we generalize a theorem of Nikita Karpenko on the parity of the first Witt index of quadratic forms to the case of a base field of characteristic two.   
\end{abstract}
\author{Olivier Haution}
\title{On the first Steenrod square for Chow groups}
\email{olivier.haution@gmail.com}
\address{Institut de Math\'ematiques de Jussieu, 175 rue du Chevaleret, 75013 Paris, France.}
\date{October 22, 2010}
\maketitle

\section{Introduction}
Good progress has been made lately towards a uniform treatment of the theory of (non-degenerate) quadratic forms, regardless of the fact the base field  has characteristic two or not. This approach is developed in the book \cite{EKM}. The main obstruction that remains, in order to accomplish this program, is that the Steenrod operations modulo two are not available when the base field has characteristic two. Indeed several constructions of the Steenrod operations for Chow groups modulo a prime number $p$ are known (\cite{EKM, Boi-A-08, Bro-St-03, Vo-03, Lev-St-05}), but none of them works over a field of characteristic $p$.

In this article we construct a weak version of the first homological Steenrod operation on modulo two Chow groups, over a field of arbitrary characteristic. More precisely, if $\Ch$ denotes the modulo two Chow group, and $\Cht$ the Chow group modulo its torsion subgroup, tensored with $\Z/2$, we construct group homomorphisms
\[
\Squu^X \colon \Ch(X) \to \Cht(X),
\]
for all separated schemes $X$ of finite type over a field. These morphisms commute with proper push-forwards, scalars extension and external product. They satisfy the formula
\[
\Squu^X[X]=\cc_1(\Tan_X)
\]
when $X$ is a smooth variety with tangent bundle $\Tan_X$, and $\cc_1$ the first Chern class.\\

Using the operation constructed here, we prove a statement about the parity of the first Witt index of a quadratic form over an arbitrary field (Theorem~\ref{th:i1}). When the base field is of characteristic different from two, this result is a particular case of a more precise statement known as Hoffmann's Conjecture, \cite[Proposition~79.4]{EKM}. But since we only construct the first Steenrod square, we only get a partial statement.\\

This work is part of my Ph.D.~thesis at the university of Paris 6 under the direction of Nikita Karpenko. I am very grateful to him for introducing me to the subject, raising the question studied here, and guiding me during this work. 

\section{Notations}
A \emph{variety} is a separated scheme of finite type over a field, which need not be irreducible, reduced nor quasi-projective. A \emph{regular variety} is a variety whose local rings are all regular local rings.\\

\paragraph{\bf Grothendieck group of coherent sheaves} Let $X$ be a variety. We shall write $\Kzl(X)$ for the Grothendieck group of coherent $\Oc_X$-sheaves. Such a sheaf $\Fc$ has a class $[\Fc] \in \Kzl(X)$, and if $f \colon X \to Y$ is a proper morphism of varieties, there is a push-forward map 
\[
f_* \colon \Kzl(X) \to \Kzl(Y) \quad , \quad [\Fc] \mapsto \sum_i (-1)^i [\mathsf{R}^if_*(\Fc)].
\]
If $f \colon X \to Y$ is flat, there is a pull-back $f^* \colon \Kzl(Y) \to \Kzl(X)$ induced by the tensor product with $\Oc_X$ over $\Oc_Y$.\\ 

\paragraph{\bf Topological filtration}
The group $\Kzl(X)$ is endowed with a topological filtration $0=\Kit{-1}{X} \subset \Kit{0}{X} \subset \cdots \subset \Kit{\dim X}{X} = \Kzl(X)$. For every integer $d$, $\Kit{d}{X}$ is the subgroup of $\Kzl(X)$ generated by classes of coherent $\Oc_X$-sheaves supported in dimension $\leq d$. This filtration is compatible with proper push-forwards. The associated graded group shall be denoted by $\Grt{\bullet}{X}$.\\

\paragraph{\bf External product} Let $X$ and $Y$ be varieties over a common field, and $p_X \colon X \times Y \to X$, $p_Y \colon X \times Y \to Y$ be the two projections. There is an external product 
\[
-\boxtimes -\colon \Kzl(X) \otimes \Kzl(Y) \to \Kzl(X \times Y)
\]
induced by the association 
\[
(\Ec,\Fc) \mapsto ({p_X}^*\Ec) \otimes_{\Oc_{X \times Y}} ({p_Y}^*\Fc).
\]
For any integers $m$ and $n$, we have (\cite[Lemma~82]{Gil-K-05})
\begin{equation*}
\Kit{m}{X} \boxtimes \Kit{n}{Y} \subset \Kit{m+n}{X \times Y}.
\end{equation*}

\paragraph{\bf Comparison with Chow groups} There is a morphism of graded abelian groups
\[
\varphi_X \colon \CH_{\bullet}(X) \to \Grt{\bullet}{X} \quad , \quad [Z] \mapsto [\Oc_Z] \mod \Kit{\dim Z-1}{X}
\]
which commutes with proper push-forwards (see \cite[Example~15.1.5]{Ful-In-98}). Let $\times$ denote the external product for Chow groups. If $X$ and $Y$ are varieties over a common field, and $x \in \Kzl(X)$, $y \in \Kzl(Y)$, we have
\begin{equation}
  \label{eq:phiext} \varphi_{X \times Y}(x \times y)=\varphi_X(x) \boxtimes \varphi_Y(y).
\end{equation}

\paragraph{\bf Scalars extension} For a variety $X$ over a field $F$, and $L/F$ a field extension, we shall write $X_L$ for the variety $X \times_{\Spec(F)} \Spec(L)$ over the field $L$. The flat morphism $X_L \to X$ induces pull-backs
\[
x \mapsto x_L, \quad \CH(X) \to \CH(X_L) \text{ and } \Kzl(X) \to \Kzl(X_L),
\]
which satisfy, for all $x \in \Kzl(X)$
\begin{equation}
  \label{eq:extension}
\varphi_{X_L}(x_L)=\varphi_X(x)_L.
\end{equation}

\section{Singular Grothendieck-Riemann-Roch Theorem}
For a variety $X$, we write $\CHQs{X}$ for $\CH(X) \otimes_{\Z} \Q$. An element $x \in \CH(X)$ has an image $x \otimes 1 \in \CHQs{X}$.

\begin{theorem}
  \label{th:RR}
  For all varieties $X$ there is a  homomorphism
  \[
  \tau^X \colon \Kzl(X) \to \CHQs{X} 
\]
with the following properties.
\begin{enumerate}[(a)]
  \item \label{proper} If $f\colon X \to Y$ is a proper morphism of varieties, we have 
    \[
    f_* \circ \tau^X=\tau^Y \circ f_*.
    \]

  \item \label{top} For an integral variety $X$, we have 
    \[
    \tau^X[\Oc_X]=[X] \mod \CHQ{<\dim X}{X}.
    \]
  \item \label{open} If $u \colon U \to X$ is an open embedding of quasi-projective varieties, then
   \[
\tau^U \circ u^*=u^* \circ \tau^X.
    \]
  \item \label{fieldext}If $X$ is a variety over a field $F$, $L/F$ a field extension, then for all $x \in \Kzl(X)$ we have in $\CHQs{X_L}$
    \[
    \tau^X(x)_L=\tau^{X_L}(x_L).
    \]
  \item \label{smooth} 
     Let $i \colon X \hookrightarrow M$ be a regular closed embedding, with $M$ a smooth variety. Let $N$ be the normal bundle of $i$ and $\Tan_M$ the tangent bundle of $M$. Then, writing $\Todd$ for the Todd class, we have
  \[
\tau^X[\Oc_X]=\Todd(i^*[\Tan_M]-[N]).
  \]
\end{enumerate}
\end{theorem}
\begin{proof}
  The existence of the map $\tau^X$ satisfying \eqref{proper} and \eqref{top} is proven in \cite[Theorem~18.3]{Ful-In-98}. Applying \emph{loc.cit.}, (4) with $f \colon X \to \mathsf{point}$ the structure morphism, and $\alpha=1 \in \Kzl(\mathsf{point})$, we obtain \eqref{smooth}. Putting $f=u\colon U \to X$, we get \eqref{open}.

 We now prove \eqref{fieldext}. We first assume that $X$ is quasi-projective over $F$, and choose a closed embedding $i \colon X \hookrightarrow M$ into a smooth $F$-variety $M$. We can assume that $x=[\Fc]$ for some coherent $\Oc_X$-sheaf $\Fc$, and take a resolution $\Ecd \to i_*\Fc \to 0$ by locally free $\Oc_M$-sheaves. Then by \cite[Theorem~18.3, (3)]{Ful-In-98}
  \[
\tau^X(x)=\ch^M_X(\Ecd) \cap \Todd(\Tan_M),
  \]
  where $\Todd(\Tan_M)$ is the Todd class of tangent bundle of $M$, and $\ch^M_X(\Ecd)\cap - \colon \CHQs{M} \to \CHQs{X}$ the localized Chern character.
 
  We have a closed embedding $i_L \colon X_L \hookrightarrow M_L$, with $M_L$ smooth over $L$. Let $f \colon X_L \to X$, $g \colon M_L \to M$ be the morphisms induced by the field extension $L/F$. Then $x_L=[f^*\Fc]$, we have an isomorphism $g^* \circ i_*\Fc \simeq (i_L)_*\circ f^*\Fc$, and a resolution $g^*\Ecd \to (i_L)_*\circ f^*\Fc \to 0$. The tangent bundle $\Tan_{M_L}$ of $M_L$ is isomorphic to $g^*\Tan_M=(\Tan_M)_L$, therefore using \cite[Theorem~18.1]{Ful-In-98}, we have
\[
\tau^{X_L}(x_L)=\ch^{M_L}_{X_L}(g^*\Ecd) \cap \Todd(g^*\Tan_{M})=f^*\big( \ch^{M}_{X}(\Ecd) \cap \Todd(\Tan_{M})\big)=\tau^{X}(x)_L.
\]

  Now we drop the assumption that $X$ is quasi-projective over $F$. Let $p \colon X' \to X$ be a Chow envelope.  The map $p_* \colon \Kzl(X') \to \Kzl(X)$ is surjective, let $y$ be an antecedent of $x$. Then we have, by compatibility with proper push-forwards
\[
  \tau^X(x)=\tau^X\circ p_*(y)=p_*\circ \tau^{X'}(y).
  \]
  The fiber product $p_L \colon {X'}_L \to X_L$ is also a Chow envelope by \cite[Lemma~18.3, (2)]{Ful-In-98}, and 
  \[
  \tau^{X_L}(x_L)= \tau^{X_L}\big( p_*(y)_L\big)=\tau^{X_L}\circ (p_L)_*(y_L)=p_* \circ \tau^{ {X'}_L}(y_L).
  \]
  The latter is equal to $p_* \circ \tau^{X'}(y)_L$ by the quasi-projective case, and~\eqref{fieldext} follows.
\end{proof}

The morphism $\tau^X$ is the \emph{homological Chern character}. Individual components
\[
\tau^X_k \colon \Kzl(X) \to \CHQ{k}{X}
\]
are defined by composing with the projections $\CHQs{X} \to \CHQ{k}{X}$.

\section{A $2$-integrality property of the homological Chern character}
The result obtained in this section, Theorem~\ref{cor:ada}, can be thought of as an algebraic analog of \cite[Theorem~1]{Ada-On-61} (with $r=1$).\\

Let $X$ be a variety. We denote by $\CHZs{X} \subset \CHQs{X}$ the Chow group of $X$ modulo its torsion subgroup, and view it as the image of the map $\CH(X) \to \CHQs{X}$ given by $x \mapsto x \otimes 1$.

\begin{lemma}
  \label{lemm:reg}
Let $X$ be a regular, connected and quasi-projective variety of dimension $d$. Then
  \[
  2\cdot\tau^X_{d-1}[\Oc_X] \in \CHZ{d-1}{X}.
  \]
\end{lemma}
\begin{proof} 
  Since $X$ is quasi-projective, we can find a smooth connected variety $M$ and a closed embedding $j \colon X \hookrightarrow M$. The morphism $j$ is a regular closed embedding, as is any closed embedding of regular varieties, (\cite[Proposition~2, \S5, N$^{\circ}$3, p.65]{Bou-AC-10}). Letting $N$ be its normal bundle and $\Tan_M$ the tangent bundle of $M$, we can apply Theorem~\ref{th:RR}, \eqref{smooth}. Using \cite[Example~3.2.4]{Ful-In-98} asserting that the first term of the Todd class is the half of the first Chern class $\cc_1$, we obtain in $\CHQs{X} \mod \CHQ{\leq \dim X -2}{X}$ the congruence
  \begin{align*}
   \tau^X[\Oc_X]=\Todd(j^*\Tan_M-N)&=\big(\id+\cc_1(j^*\Tan_M) \otimes 2^{-1}\big)\cdot\big(\id+\cc_1(N)\otimes 2^{-1}\big)^{-1}[X]\\
   &=[X] + \cc_1(j^*\Tan_M - N) \otimes 2^{-1}.
 \end{align*}
The statement follows.
\end{proof}

\begin{lemma}
  \label{lemm:normal}
  Let $X$ be a normal, connected and quasi-projective variety of dimension $d$. Then
  \[
  2\cdot \tau^X_{d-1}[\Oc_X] \in \CHZ{d-1}{X}.
  \]
 \end{lemma}
\begin{proof} 
The set $S$ of points $x \in X$ such that $\Oc_{X,x}$ is not a regular local ring is closed in $X$ (\cite[Corollaire~4, \S7, N$^{\circ}$9, p.102]{Bou-AC-10}). We consider $S$ as a closed subscheme of $X$, by endowing it with the reduced scheme structure. Since $X$ is normal, the subscheme $S$ has codimension at least two in $X$.

Let $u\colon U \to X$ be the open complement of $S$ in $X$. The variety $U$ is regular, connected and quasi-projective, hence Lemma~\ref{lemm:reg} applies to $U$: we find an integral cycle $y_U \in \CH_{d-1}(U)$ satisfying $2\cdot\tau^U_{d-1}[\Oc_U]=y_U \otimes 1$. We have the localization sequence 
\[
\CH(S) \to \CH(X) \xrightarrow{u^*} \CH(U) \to 0.
\]
Let $y \in \CH_{d-1}(X)$ be such that $u^*(y)=y_U$. By Theorem~\ref{th:RR}, \eqref{open}, we have 
\[
u^* (2\cdot \tau^X_{d-1}[\Oc_X]-y \otimes 1)=2\cdot \tau^U_{d-1}[\Oc_U]-y_U \otimes 1=0,
\]
hence $2\cdot \tau^X_{d-1}[\Oc_X]-y \otimes 1$ belongs to the image of $\CHQs{S} \to \CHQs{X}$, which is contained in $\CHQ{<d-1}{X}$. But $2\cdot\tau^X_{d-1}[\Oc_X]-y\otimes 1$ belongs to $\CHQ{d-1}{X}$, hence it is zero. 
\end{proof}

\begin{proposition}
   \label{prop:ada1}
 Let $X$ be an integral variety of dimension $d$. Then
  \[
  2\cdot \tau^X_{d-1}[\Oc_X] \in \CHZ{d-1}{X}.
  \]
\end{proposition}
\begin{proof}
  First, using Chow's Lemma \cite[II, Th\'eor\`eme~5.6.1]{ega-2}, choose a projective birational morphism $e \colon X' \to X$, with $X'$ quasi-projective and integral. Let $n \colon \ti{X} \to X'$ be the normalization of $X'$. It is a finite birational morphism (\cite[Th\'eor\`eme~2, Chapitre~V, \S3, N$^\circ$2, p.59]{Bou-Ac-47}), with $\ti{X}$ normal, connected, and quasi-projective. Setting $p=e \circ n$, we have $p_*[\Oc_{\ti{X}}]=[\Oc_X]+x$ for some element $x \in \Kit{d-1}{X}$. Then 
  \[
  2\cdot \tau^X_{d-1}[\Oc_X]=2\cdot p_*\circ \tau^{\ti{X}}_{d-1}[\Oc_{\ti{X}}]-2\cdot \tau^X_{d-1}(x).
  \]
  We conclude by applying Lemma~\ref{lemm:normal} to $\ti{X}$, and using Lemma~\ref{lemm:dim} below.
\end{proof}

\begin{lemma}
  \label{lemm:dim}
  Let $X$ be a variety and $x \in \Kit{k-1}{X}$. Then 
\[
 \tau^X_{k-1}(x) \in \CHZ{k-1}{X}.
\]
\end{lemma}
\begin{proof}
  The group $\Kit{k-1}{X}$ is generated by the elements $i_*[\Oc_Z]$ with $Z$ an integral variety of dimension $\leq k-1$, and $i \colon Z \hookrightarrow X$ a closed embedding. We can assume that $x$ is such a generator. Then by Theorem~\ref{th:RR}, \eqref{proper} we have $\tau^X_{k-1}(x)=i_* \circ \tau^Z_{k-1}[\Oc_Z]$. In view of Theorem~\ref{th:RR}, \eqref{top}, this element is either zero (when $\dim Z < k-1$) or $i_*[Z]\otimes 1$ (when $\dim Z=k-1$). Therefore it belongs to $\CHZ{k-1}{X}$.
\end{proof}

\begin{theorem}
  \label{cor:ada}
  Let $X$ be a variety, and $x \in \Kit{k}{X}$. Then 
  \[
  2 \cdot \tau^X_{k-1}(x) \in \CHZ{k-1}{X}.
  \]
 \end{theorem}
 \begin{proof}
   The group $\Kit{k}{X}$ is generated by the subgroup $\Kit{k-1}{X}$ and elements of type $i_*[\Oc_Z]$ with $i \colon Z \hookrightarrow X$ an integral closed subvariety of dimension $k$. Applying Lemma~\ref{lemm:dim}, we can assume that $x$ is an element of this type. Then by Theorem~\ref{th:RR}, \eqref{proper}, we have $2 \cdot \tau^X_{k-1}(x)=i_* (2 \cdot \tau^Z_{k-1}[\Oc_Z])$, which belongs to $i_*\CHZ{k-1}{Z} \subset \CHZ{k-1}{X}$ because of Proposition~\ref{prop:ada1}. 
 \end{proof}

\section{The first homological Steenrod square}
We define functors from the category of varieties and proper morphisms to the category of graded abelian groups by setting, for every variety $X$, $\Ch(X)=  \Z/2\otimes\CH(X)$ and 
\[
\Cht(X)=\CH(X)/(2\text{-torsion}+ 2 \CH(X))=\Z/2 \otimes \CHZs{X}.
\]
There is a natural surjective map $\Ch(X) \to \Cht(X)$, $x \mapsto \ti{x}$.\\

Let $X$ be a variety. We define a homomorphism of abelian groups
\[
s_k \colon \Kit{k}{X} \to \CHZ{k-1}{X} \to \Cht_{k-1}(X),
\]
the first map being the association $x \mapsto 2\cdot \tau^X_{k-1}(x) \in \CHQ{k-1}{X}$, which has image in the subgroup $\CHZ{k-1}{X}$ by Theorem~\ref{cor:ada}, and the second map being the reduction modulo two. 

Now take $x \in \Kit{k-1}{X} \subset \Kit{k}{X}$. By Lemma~\ref{lemm:dim}, we have $2\cdot\tau^X_{k-1}(x) \in 2\cdot\CHZ{k-1}{X}$, hence $s_k(x)=0 \in  \Cht_{k-1}(X)$. This gives a homomorphism of graded abelian groups
\[
\Z/2 \otimes \Grts(X) \to \Cht_{\bullet-1}(X),
\]
which commutes with proper push-forwards by Theorem~\ref{th:RR}, \eqref{proper}.

We compose with the natural transformation $\id_{\Z/2} \otimes \varphi  \colon \Ch_\bullet \to \Z/2 \otimes \Grts$, and obtain a natural transformation of functors from the category of varieties and proper morphisms to the category of graded abelian groups
\[
\Squu \colon \Ch_\bullet \to \Cht_{\bullet-1}.
\]
For a variety $X$, we write $\Squu^X \colon \Ch_\bullet(X) \to \Cht_{\bullet-1}(X)$ for the induced morphism.

\begin{remark}
  \label{rem:sq1descend}
One can expect the map $\Squu$ to lift to a morphism 
\[
\Ch_\bullet \to \Ch_{\bullet -1}.
\]
However one cannot expect that $\Squu$ descends to a morphism
\[
\Cht_\bullet \to \Cht_{\bullet -1},
\]
as suggested by the following example.

Let $X$ be an anisotropic projective $3$-dimensional quadric over a field $F$ of characteristic not two, defined by a quadratic form of type $\langle \langle a,b \rangle \rangle\perp c$. In this case, by~\cite[Theorem~5.3]{Kar-90} there is an element $l_0 \in \Kit{1}{X}$ such that $l_0 \notin \Kit{0}{X}$ and $2l_0 \in \Kit{0}{X}$. Note that $\varphi_X \colon  \CH_\bullet(X) \to \Grts(X)$ is an isomorphism : this is a general fact concerning smooth varieties in codimension $0,1$ and $2$ (\cite[Example~15.3.6]{Ful-In-98}); in codimension $3$ this follows from the fact that $\CH_0(X) \simeq \Z$ (\cite[Corollary~71.4]{EKM}) and the fact that $\varphi_X$ has torsion kernel (\cite[Example~15.3.6]{Ful-In-98}). Let $x \in \Ch_1(X)$ be the antecedent of $l_0 \mod \Kit{0}{X}$ under $\varphi_X$, and $y \in \Ch_0(X)$ the class of a point of degree two. We have $\varphi_X(y)=2 l_0$, and using Theorem~\ref{th:RR}, \eqref{top}
\[
\tau^X(l_0)=2^{-1} \tau^X(2l_0)=y \otimes 2^{-1}
\]
hence $\Squu^X(x)=y$, which is non-zero in $\Cht_0(X)$. On the other hand $x$ is zero in $\Cht_1(X)$.
\end{remark}

\begin{proposition}
  \label{prop:extsquu}
Let $X$ and $Y$ be two varieties over the same field, $x \in \CH_n(X)$, and $y \in \CH_m(Y)$. We have  
\[
\Squu^{X \times Y}(x \times y)=\ti{x} \times \Squu^Y(y) + \Squu^X(x) \times \ti{y}.
\]
\end{proposition}
\begin{proof}
We know by Theorem~\ref{th:RR},~\eqref{top} that $\tau^X_k(x)=0$ and $\tau^Y_l(y)$ for $k>n$ and $l>m$. Then, using \cite[Example~18.3.1]{Ful-In-98} and \eqref{eq:phiext}, we compute
  \begin{align*}
    &2 \cdot \tau^{X \times Y}_{m+n-1} \circ \varphi_{X \times Y}(x \times y)\\
    =&2 \cdot \tau^{X \times Y}_{m+n-1}\big(\varphi_X(x) \boxtimes \varphi_Y(y)\big)\\
  =&\tau_m^X\circ \varphi_X(x) \times 2\cdot \tau_{n-1}^Y\circ \varphi_Y(y) + 2\cdot \tau_{m-1}^X\circ \varphi_X(x) \times \tau_n^Y\circ \varphi_Y(y)\\
  =&(x \otimes 1) \times 2\cdot \tau_{n-1}^Y\circ \varphi_Y(y) + 2\cdot \tau_{m-1}^X\circ \varphi_X(x) \times (y\otimes 1). \qedhere
\end{align*}
\end{proof}

\begin{proposition}
  \label{prop:wu}
  Let $X$ be a smooth variety, with tangent bundle $\Tan_X$. Then 
\[
\Squu^X[X]=\ti{\cc_1(\Tan_X)}.
\]
\end{proposition}
\begin{proof}
  This follows from the proof of Lemma~\ref{lemm:reg}.
\end{proof}

\begin{remark}[cohomological operation]
  \label{rem:coh}
For a smooth variety $X$, define 
\[
\Squ_X \colon \Ch^\bullet(X) \to \Cht^{\bullet +1}(X), \quad x \mapsto \Squu^X(x) + \ti{\cc_1(\Tan_X)}\cdot \ti{x}.
\]
One can prove that $\Squ$ commutes with pull-backs along arbitrary morphisms of smooth varieties. The formula $\Squ_X(x\cdot y)=\ti{x} \cdot \Squ_X(y) + \Squ_X(x) \cdot \ti{y}$ then follows from Proposition~\ref{prop:extsquu}. Using  the fact that the group $\CH^1(X)$ is generated by regularly embedded subvarieties when $X$ is smooth, we get the formula $\Squ_X(x)=\ti{x}^2$ for $x \in \Ch^1(X)$. 

This operation would have have been sufficient to prove Theorem~\ref{th:i1} below, and could also have been constructed directly using a simpler form of the Riemann-Roch theorem instead of Theorem~\ref{th:RR}. However it is not clear how to reconstruct the homological operation on singular varieties from the cohomological one.
\end{remark}

We now consider a variety $X$ over a field $F$, and a field extension $\ba{F}/F$ such that variety $\ba{X}=X_{\ba{F}}$ has a torsion-free Chow group. Examples of such varieties include those $X$ such that $\ba{X}$ is cellular. In particular $X$ could be a projective homogeneous variety under a semi-simple algebraic group, and $\ba{F}/F$ an algebraic closure. 

Since $\Squu$ commutes with field extensions by \eqref{eq:extension} and Theorem~\ref{th:RR}, \eqref{fieldext}, we have a commutative diagram, where vertical arrows are pull-backs along the flat morphism $\ba{X} \to X$,
\[
\xymatrix{ 
\Ch(X) \ar[rr]^{\Squu^X} \ar[d] && \Cht(X) \ar[d]\\ 
\Ch(\ba{X}) \ar[rr]_{\Squu^{\ba{X}}} && \Cht(\ba{X}). 
  }
  \]
  Note that $\Ch(\ba{X})\simeq \Cht(\ba{X})$, hence the operation $\Squu^{\ba{X}}$ is an endomorphism of $\Ch(\ba{X})$ which preserves rationality, \emph{i.e.\ }induces an endomorphism of the image of the map $\Ch(X) \to \Ch(\ba{X})$.

\section{Parity of the first Witt index of a quadratic form}
In this section we prove a consequence of Hoffmann's Conjecture, over a field of arbitrary characteristic. The proof is modeled upon the method used in \cite{Kar-On-2003}.\\

Let $X$ be a smooth projective quadric over a field $F$, and $\ba{F}/F$ a splitting field extension for $X$, \emph{i.e.\ }an extension such that the quadric $\ba{X}=X_{\ba{F}}$ contains a projective space of the maximal possible dimension $d$ (the number $d$ is the greatest integer such that $2d \leq \dim X$). The variety $X$ and the field extension $\ba{F}/F$ satisfy the assumption of the previous section. Indeed a basis for the free abelian group $\CH(\ba{X})$ is given by elements $h^i, l_i$ with $0 \leq i \leq d$, see for example \cite[Proposition~68.1]{EKM}. The cycle $h\in \CH^1(\ba{X})$ is the hyperplane class, and $l_i \in \CH_i(\ba{X})$ is the class of a projective subspace of dimension $i$.

A cycle in $\Ch(\ba{X}\times \ba{X})$ is said to be \emph{rational} if it is the restriction of a cycle in $\Ch(X\times X)$.

\begin{lemma}
 \label{lemm:compute_Sq1}
 Let $X$ be a smooth projective quadric, and $\ba{X}$ the restriction of $X$ to a splitting field. Then we have in $\Ch_{i-1}(\ba{X})$
  \[
  \Squu^{\ba{X}}(l_i) = (i+1)\cdot l_{i-1}.
  \]
\end{lemma}
\begin{proof} 
  Let $j \colon \PP^i \hookrightarrow \ba{X}$ be a closed embedding representing the cycle $l_i$. Then since $j_* \circ \Squu^{\PP^i}=\Squu^{\ba{X}} \circ j_*$, and by Proposition~\ref{prop:wu} we have
  \[
  \Squu^{\ba{X}}(l_i) = j_* \circ \cc_1(\Tan_{\PP^i}).
  \]
  By \cite[Example~104.20]{EKM}, this is equal to 
  \[
  j_* \circ \cc_1\big( (i+1)\cdot [\Oc(1)] -[\Oc_{\PP^i}]\big)= j_* \big( (i+1)\cdot [\PP^{i-1}]\big)=(i+1) \cdot l_{i-1}.\qedhere
  \]
    \end{proof}

\begin{theorem}
  \label{th:i1}
Let $\varphi$ be an anisotropic non-degenerate quadratic form over an arbitrary  field.  Let $\iu$ be the first Witt index of $\varphi$. If $\dim \varphi - \iu $ is odd then $\iu=1$.
\end{theorem}
\begin{proof}
  Let $F$ be the ground field, $X$ be the smooth projective quadric of $\varphi$, and $\ba{F}/F$ a splitting field extension for $X$. We assume that $\iu \neq 1$. We use the notion of a cycle contained in another (\cite[p.313]{EKM}). Let $\pi \in \Ch(\ba{X}\times \ba{X})$ be the $1$-primordial cycle (\cite[p.323]{EKM}). It is the ``minimal'' (\cite[Definition~73.5]{EKM}) rational cycle containing $h^0 \times l_{\iu-1}$. Write 
  \[
  \pi=h^0 \times l_{\iu-1}+l_{\iu-1} \times h^0 + v.
  \]
  Then $v \in \Ch(\ba{X}\times \ba{X})$ does not contain $h^0 \times l_{\iu-1}$ nor $l_{\iu-1} \times h^0$ by \cite[Lemma~73.15]{EKM}. Also the rational cycle $\pi$, hence $v$, does not contain $h^1 \times l_{\iu}$ nor $l_{\iu} \times h^1$ (these points lie outside of the ``shell triangles'', \emph{i.e.\ }are forbidden by \cite[Lemma~73.12]{EKM}). Using Proposition~\ref{prop:extsquu}, we see that the cycle $\Squu^{\ba{X}\times \ba{X}}(v)$ does not contain any of the cycles $h^0 \times l_{\iu-2}$, $h^1 \times l_{\iu-1}$, $l_{\iu-2}\times h^0$, $l_{\iu-1} \times h^1$.\\
 
  Next we compute $\Squu^{\ba{X}\times \ba{X}}(\pi)$, which is rational by the considerations of the previous section.  We have, by Proposition~\ref{prop:extsquu}, Lemma~\ref{lemm:compute_Sq1} and Proposition~\ref{prop:wu}, 
  \[
  \Squu^{\ba{X}\times \ba{X}}(h^0 \times l_{\iu-1})=h^0 \times \Squu^{\ba{X}}(l_{\iu-1}) + \cc_1(\Tan_X) \times l_{\iu-1}=\iu \cdot (h^0 \times l_{\iu-2}) + \cc_1(\Tan_X) \times l_{\iu-1}.
  \]
By \cite[Lemma~78.1]{EKM}, we know that the modulo two total Chern class of the tangent bundle $\Tan_X$ is $(1+h)^{\dim X +2}$, hence $\cc_1(\Tan_X)=(\dim X + 2)\cdot h^1$. This gives 
\[
\Squu^{\ba{X}\times \ba{X}}(\pi)= \iu \cdot (h^0 \times l_{\iu-2}+l_{\iu-2} \times h^0) + (\dim \varphi)\cdot (h^1 \times l_{\iu-1}+l_{\iu-1} \times h^1) + \Squu^{\ba{X}\times \ba{X}}(v). 
\]
Now by \cite[Corollary~73.21]{EKM} and the property of $\Squu^{\ba{X}\times \ba{X}}(v)$ mentioned above, we see that $\dim \varphi$ and $\iu$ must have the same parity.
\end{proof}

\bibliographystyle{alpha}

\end{document}